\documentclass[leqno,11pt]{amsart}
\usepackage[colorlinks,pagebackref,hypertexnames=false]{hyperref}

\usepackage{tikz}
\usepackage{pgfplots}
\pgfplotsset{width=10cm,compat=1.9}
\usepgfplotslibrary{colorbrewer,patchplots}

\usepackage{float}
\usepackage{amsmath,amsthm,amssymb,mathrsfs} 
\usepackage[alphabetic,backrefs]{amsrefs}
\usepackage{ae,aecompl}
\usepackage[margin=1.2in]{geometry} 
\usepackage{caption}

\usepackage{cite}

\usepackage[english]{babel}

\usepackage{bookmark}

\numberwithin{equation}{section}
\usepackage{microtype}

\usepackage{fancybox,fancyhdr,graphics,epsfig}

\numberwithin{equation}{section}

\newcommand{\rG}{{\rm G}}

\newcommand{\bA}{{\bf A}}

\newcommand{\bF}{{\bf F}}

\newcommand{\cF}{\mathcal{F}}

\newcommand{\cH}{\mathcal{H}}

\newcommand{\cO}{\mathcal{O}}

\newcommand{\SO}{{\rm SO}}
\newcommand{\G}{{\rm G}}
\newcommand{\Sp}{{\rm Sp}}
\newcommand{\Spin}{{\rm Spin}}
\newcommand{\SU}{{\rm SU}}

\newcommand{\U}{{\rm U}}

\renewcommand{\epsilon}{\varepsilon}

\renewcommand{\Im}{\mathop{\mathrm{Im}}}

\renewcommand{\Re}{\mathop{\mathrm{Re}}}

\newcommand{\vol}{\mathrm{vol}}

\def\<{\mathopen{}\left<}
\def\>{\right>\mathclose{}}
\def\({\mathopen{}\left(}
\def\){\right)\mathclose{}}

\newtheorem{theorem}{Theorem}
\newtheorem{proposition}[theorem]{Proposition}

\newtheorem{corollary}[theorem]{Corollary}
\newtheorem{lemma}[theorem]{Lemma}
\theoremstyle{definition}

\newtheorem{definition}[theorem]{Definition}
\newtheorem{example}[theorem]{Example}

\newtheorem{remark}[theorem]{Remark}

\numberwithin{equation}{section}

\newtheorem*{theorem*}{Theorem}
\newtheorem*{prop*}{Proposition}
\newtheorem*{corol*}{Corollary}

\numberwithin{theorem}{section}

\author{Jason D. Lotay}
\address{University of Oxford, U.K.}
\urladdr{\href{http://people.maths.ox.ac.uk/lotay/}{http://people.maths.ox.ac.uk/lotay/}}
\email{jason.lotay@maths.ox.ac.uk}

\author{Gon\c calo Oliveira}
\address{Universidade Federal Fluminense, IME-GMA, Niter\'oi, Brazil}
\urladdr{\href{https://sites.google.com/view/goncalo-oliveira-math-webpage/home}{https://sites.google.com/view/goncalo-oliveira-math-webpage/home}}
\email{galato97@gmail.com}

\title[Examples of deformed G\textsubscript{2}-instantons]{Examples of deformed G\textsubscript{2}-instantons/Donaldson--Thomas connections}

\begin{document}

\begin{abstract}
In this note, we provide the first non-trivial examples of deformed $\G_2$-instantons, originally called deformed Donaldson--Thomas connections.  As a consequence, we see how deformed $\G_2$-instantons can be used to 
distinguish between nearly parallel $\G_2$-structures and isometric $\G_2$-structures on 3-Sasakian 7-manifolds. Our examples give non-trivial deformed $\G_2$-instantons with obstructed deformation theory and situations where the moduli space of deformed $\G_2$-instantons has components of different dimensions.  We finally study the relation between our examples and a Chern--Simons type functional which has deformed $\G_2$-instantons as critical points.
\end{abstract}

\maketitle

\tableofcontents

%===============================================================================

%===============================================================================

%===============================================================================

\section{Introduction}

Gauge theory and calibrated geometry are central topics of study in the context of $\G_2$ geometry, and they are intimately related.  
Based on ideas stemming from Mirror Symmetry, particularly SYZ fibrations, and the real Fourier--Mukai transform, the authors of \cites{LeeLeung} introduced the following gauge-theoretic equation (in the context of complex line bundles over $\G_2$-manifolds)  as a proposed mirror to certain calibrated cycles in $\G_2$-manifolds.  

\begin{definition}
Let $(X^7,\varphi)$ be a 7-manifold with a coclosed $\G_2$ structure $\varphi$, let $\psi=*_{\varphi}\varphi$ be the dual of $\varphi$, and let $L$ be a Hermitian complex line bundle on $X$.  A unitary connection $A$ on $L$ is a \emph{deformed $\G_2$-instanton} if its curvature $F_A$ satisfies
\begin{equation}\label{eq:dG2.intro}
\frac{1}{6}F_A^3+F_A\wedge \psi=0.
\end{equation}
The definition can obviously be extended to higher rank vector bundles and principal bundles but,
 based on \cites{LeeLeung}, one is primarily interested in the case of complex line bundles. When $(X,\varphi)$ is additionally a $\G_2$-manifold, deformed $\G_2$-instantons are, in a certain sense, ``mirror'' to (co)associative cycles.
\end{definition}

\begin{remark}
In \cites{LeeLeung}, deformed $\G_2$-instantons are called \emph{deformed Donaldson--Thomas connections}.  However, since $A$ is a \emph{$\G_2$-instanton} on $(X^7,\varphi)$ if and only if
\begin{equation}\label{eq:G2inst.intro}
F_A\wedge\psi=0,
\end{equation}
the authors feel it is more appropriate that solutions of \eqref{eq:dG2.intro} are called deformed $\G_2$-instantons.  Moreover, there is a natural relationship between deformed $\G_2$-instantons and deformed Hermitian-Yang--Mills connections, which is parallel to the relationship between $\G_2$-instantons and Hermitian-Yang--Mills connections, that gives further rationale for the nomenclature.
\end{remark}

\subsection{Main results} The main results of this article are the first constructions of non-trivial solutions to the deformed $\G_2$-instanton equation \eqref{eq:dG2.intro}. Here, by non-trivial, we mean deformed $\G_2$-instantons that are not flat and do not arise via pullback from lower-dimensional constructions. Our construction takes place on compact $7$-manifolds equipped with $\rG_2$-structures related to the existence of a $3$-Sasakian structure. This is a setting where interesting families of $\rG_2$-structures can be found and at the same time symmetries can be used to turn the problem of finding deformed $\rG_2$-instantons into a more tractable one. 

As an application of our construction we have the following result (Corollary  \ref{cor:examples.np}).

\begin{theorem}\label{thm:np.intro}
Let $X^7$ be a compact $3$-Sasakian $7$-manifold and let $L_0$ be the trivial complex line bundle on $X$.   Let $\varphi^{ts}$ be the standard nearly parallel $\G_2$-structure on $X$ inducing the $3$-Sasakian Einstein metric $g^{ts}$, and let $\varphi^{np}$ be the second (strictly) nearly parallel $\G_2$-structure on $X$ inducing the ``squashed'' Einstein metric $g^{np}$.
\begin{itemize}
\item There is a circle of non-trivial deformed $\G_2$-instantons  on $L_0$ for $\varphi^{ts}$.
\item There is a $2$-sphere of non-trivial deformed $\G_2$-instantons on $L_0$ for $\varphi^{np}$.
\end{itemize} 
\end{theorem}

\begin{remark}  There are infinitely many compact 
$3$-Sasakian $7$-manifolds $X^7$ \cites{BGM,GWZ}.  Key examples are giving by the $7$-sphere $S^7$ and the Aloff--Wallach space $\SU(3)/\U(1)_{1,1}$ (c.f.~Example \ref{ex:AW}), which is the $\SO(3)$-frame bundle of $\Lambda^2_+\overline{\mathbb{CP}^2}$. 
\end{remark}

\begin{remark} We recall that, if $X^7$ is a compact $3$-Sasakian $7$-manifold, then the metric cone on $(X^7,g^{ts})$ is hyperk\"ahler, and has holonomy $\mathrm{Sp}(2)$ if it is not flat, whilst the metric cone on $(X^7,g^{np})$ has holonomy $\mathrm{Spin}(7)$.
\end{remark}

\begin{remark}
Theorem \ref{thm:np.intro} shows how non-trivial deformed $\G_2$-instantons distinguish between the nearly parallel $\G_2$-structures $\varphi^{np}$ and $\varphi^{ts}$.  We shall also see  that non-trivial deformed $\G_2$-instantons can discriminate between two coclosed $\G_2$-structures inducing the same metric, including the Einstein metrics $g^{ts}$ and $g^{np}$ (c.f.~Corollary \ref{cor:examples.metrics}).  
\end{remark}

We can also apply our results to non-trivial complex line bundles when $X^7$ is the 3-Sasakian Aloff--Wallach space 
(Corollary \ref{cor:AW}).

\begin{theorem}\label{cor:AW_Intro}
	Let $\pi:X\to \overline{\mathbb{CP}^2}$ be the $\SO(3)$-frame bundle of $\Lambda^2_+ \overline{\mathbb{CP}^2}$ and let $k\in\mathbb{Z}$.
	\begin{itemize}
	\item	There is a circle of non-trivial deformed $\G_2$-instantons on $\pi^*\mathcal{O}(k)$ for $\varphi^{ts}$.
\item	There is a $2$-sphere of non-trivial deformed $\G_2$-instantons on $\pi^*\mathcal{O}(k)$ for $\varphi^{np}$.
	\end{itemize}
\end{theorem}

More generally, we give examples of deformed $\G_2$-instantons for families of coclosed $\G_2$-structures $\varphi_{t,\varepsilon}$ on a compact 3-Sasakian 7-manifold $X$ depending on two parameters: $t>0$ and $\varepsilon\in\{\pm 1\}$.  
Our ansatz depends on $a_1,a_2,a_3\in\mathbb{R}$, where $a_1=a_2=a_3=0$ yields the trivial flat connection in the case of the trivial complex line bundle $L_0$.  Hence, if we let $r=\sqrt{a_1^2+a_2^2+a_3^2}$, which can be viewed as the distance to the trivial connection on $L_0$, we can represent our main result (Proposition \ref{prop:examples}) in Figure \ref{fig:moduli} below.   The overall picture for the non-trivial line bundles on the 3-Sasakian Aloff--Wallach space (Proposition \ref{prop:AW}) is the same.
   \begin{center}
\begin{figure}[ht]
\begin{tikzpicture}[scale=1]
\node[draw] at (1,3.5) {$\varepsilon=+1$};
\draw[->,thick] (0,0) -- (0,4); 
\draw[->,thick] (0,0) -- (4,0);
\draw[red,thick] (0,0)-- (4,0);
\draw[blue,thick] (4*0.707,0) arc(0:90:4*0.707 and 4*0.5);
\draw (4,-0.3) node {$t$};
\draw (-0.2,4) node {$r$};
\draw[blue] (2,1.8) node {$S^2$};
\draw[red] (2,0.2) node {point};
\draw (4*0.707,-0.4) node {$\textstyle\frac{1}{\sqrt{2}}$};
\end{tikzpicture}\qquad\qquad
\begin{tikzpicture}[scale=1]
\node[draw] at (1,3.5) {$\varepsilon=-1$};
\draw[->,thick] (0,0) -- (0,4); 
\draw[->,thick] (0,0) -- (4,0);
\draw[red,thick] (0,0)-- (4,0);
\draw[blue,thick] (4*0.707,0) arc(0:90:4*0.707 and 4*0.5);
\draw[green,thick, variable=\x, domain=0:4] plot ({\x}, {4*1/2*sqrt(1+2*\x/4*\x/4)});
\draw (4,-0.3) node {$t$};
\draw (-0.2,4) node {$r$};
\draw[blue] (2.3,1.8) node {2 points};
\draw[red] (2,0.2) node {point};
\draw[green] (2,2.8) node {$S^1$};
\draw (4*0.707,-0.4) node {$\textstyle\frac{1}{\sqrt{2}}$};
\end{tikzpicture}
\caption{The space of deformed $\G_2$-instantons for $\varphi_{t,\varepsilon}$ on $L_0$.}\label{fig:moduli}
\end{figure}
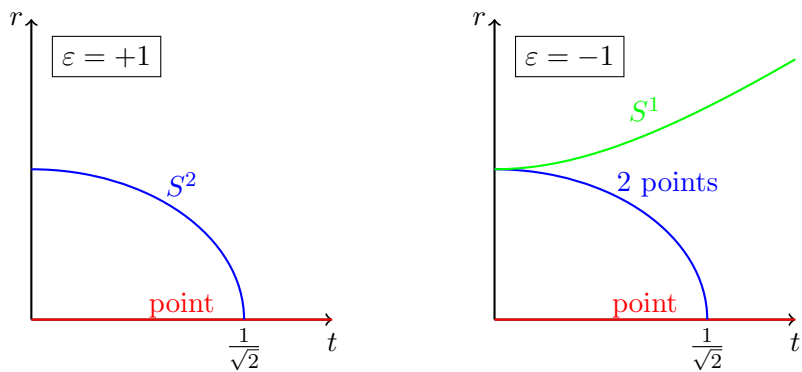
\end{center}
%\vspace{-20pt}
\begin{remark}\label{rem:moduli} In particular, we observe in Figure \ref{fig:moduli} how deformed $\G_2$-instantons die as $t$ varies from $0$ to $1/\sqrt{2}$ in both the cases $\varepsilon=\pm 1$, but when $\varepsilon=-1$ we have a circle that lives for all $t>0$.  For context, we mention that $\varphi^{np}$ corresponds to $(t,\varepsilon)=(1/\sqrt{5},+1)$ and $\varphi^{ts}$ corresponds to $(t,\varepsilon)=(1,-1)$.  We also note that $\varphi_{t,\varepsilon}$ induce the same metric for a fixed value of $t$; it is an interesting feature resulting from our work that the space of deformed $\G_2$-instantons differs for these isometric $\G_2$-structures.
\end{remark}

\begin{remark}
 The parameter $t$ in Figure \ref{fig:moduli} corresponds to the size of the fibres for the canonical $\SU(2)$ or $\SO(3)$ fibration of the 3-Sasakian 7-manifold $X$ over a 4-orbifold.  We see that we have non-trivial deformed $\G_2$-instantons which tend to a non-trivial limit connection as $t\to 0$ (i.e.~as the 3-dimensional fibres collapse), even though the norms of their curvatures are blowing up (with respect to the $t$-dependent metric).  This suggests possible links of our study to adiabatic limits and compactness issues for deformed $\G_2$-instantons.
\end{remark}

We may also relate our examples to the recently developed moduli space theory for deformed $\G_2$-instantons in \cites{KawaiYamamoto}.  We refer the reader to Definition \ref{dfn:moduli} for the formal definition for a deformed $\G_2$-instanton to be obstructed in the sense of deformation theory.

\begin{theorem}  Let $X^7$ be a compact $3$-Sasakian $7$-manifold and let $L_0$ be the trivial complex line bundle on $X$. 
The non-trivial deformed $\G_2$-instantons given by Theorem \ref{thm:np.intro} are obstructed.  Moreover, the moduli spaces of deformed $\G_2$-instantons on $L_0$ for the nearly parallel $\G_2$-structures $\varphi^{np}$ and $\varphi^{ts}$ both contain at least two components of different dimensions.
\end{theorem}

A Chern--Simons type functional was introduced in \cites{KarigiannisLeung} whose critical points are deformed $\G_2$-instantons.  We study this functional in our setting and, in particular, deduce the following (c.f.~Lemma \ref{lem:trivial.moduli} and Corollary \ref{cor:flat.obs}), which reflects the fact that non-trivial deformed $\G_2$-instantons coalesce with the trivial connection at $t=1/\sqrt{2}$. 

\begin{theorem}\label{thm:flat.obs}
Let $X^7$ be a compact $3$-Sasakian $7$-manifold, let $L_0$ be the trivial complex line bundle, and let $A_0$ be the trivial flat connection on $L_0$.  Then $A_0$ is unobstructed (and hence rigid and locally isolated) as a deformed $\G_2$-instanton with respect to 
$\varphi^{np}=\varphi_{1/\sqrt{5},+1}$ and $\varphi^{ts}=\varphi_{1,-1}$, but obstructed as a deformed $\G_2$-instanton with respect to $\varphi_{1/\sqrt{2},\varepsilon}$ for $\varepsilon=\pm 1$.
\end{theorem}

\subsection{Summary}

This article is organized as follows. Section \ref{sec:3Sasakian} introduces some background on $3$-Sasakian geometry which will later be of use in constructing the examples of deformed $\G_2$-instantons. In section \ref{sec:Pullback} we give some simple examples of deformed $\G_2$-instantons which arise from pulling back connections in 6 and 4 dimensions. The non-trivial examples which constitute the main contribution of this article are constructed and presented in section \ref{sec:Examples}. Finally, in section \ref{sec:Chern-Simons}, this article concludes with a discussion of the Chern--Simons type functional mentioned above;
in particular, the functional is explicitly computed and analyzed in some of the cases developed in this article.

\begin{remark}
We shall use summation convention over repeated indices throughout the article.
\end{remark}

\subsection*{Acknowledgments}
The first author was supported by the Simons Collaboration on Special Holonomy in Geometry, Analysis, and Physics (\#724071 Jason Lotay).  The second author was supported by Funda\c{c}\~ao Serrapilheira 1812-27395, by CNPq grants 428959/2018-0 and 307475/2018-2, and FAPERJ through the program Jovem Cientista do Nosso Estado E-26/202.793/2019. The authors thank Conan Leung for suggesting the possible link of this work to adiabatic limits. In addition, the authors would like to thank Paul-Andi Nagy and an anonymous referee for their helpful comments on the article.

\section{3-Sasakian geometry}\label{sec:3Sasakian}

We shall now give a short introduction to some aspects of $3$-Sasakian geometry and its relation to $\G_2$ geometry. We refer the reader to the survey article \cites{Boyer2001} and references therein for more information on $3$-Sasakian geometry. 

\begin{definition}\label{dfn:3Sasakian}
A complete Riemannian $7$-manifold $(X^7, g^{ts})$ is \emph{$3$-Sasakian} if it has three orthonormal Killing vector fields $\lbrace \xi_i \rbrace_{i=1}^3$ satisfying the relations $[\xi_i, \xi_j] = 2\epsilon_{ijk} \xi_k$, for $\epsilon_{ijk}$ the sign of the permutation taking $(1,2,3)$ to $(i,j,k)$, such that each $\xi_i$ induces a Sasakian structure on $(X^7, g^{ts})$.
\end{definition}

As a consequence of this definition, for a $3$-Sasakian $7$-manifold $(X^7, g^{ts})$, the complete metric $g^{ts}$ is Einstein with positive scalar curvature and its metric cone has holonomy contained in $\Sp(2)$. In particular, $X^7$ is compact by Myers' theorem.

\subsection{SU(2) leaf space} For a $3$-Sasakian manifold $(X^7,g^{ts})$ as in Definition \ref{dfn:3Sasakian}, the Killing vector fields $\lbrace \xi_i \rbrace_{i=1}^3$ generate a locally free action of   $\SU(2)$. The leaf space of this $\SU(2)$ action, denoted $Z^4$, can be endowed with a metric $g_Z$ so that the canonical projection 
 \begin{equation}\label{eq:proj.XZ}
 \pi:X \rightarrow Z
 \end{equation}
  is an orbifold Riemannian submersion. This metric is anti-self-dual and Einstein of positive scalar curvature $s>0$.  For convenience, we will scale the metric $g^{ts}$ so that $s=48$: this fits with the canonical example of $S^7$ with its constant curvature 1 metric. In the particular case when $Z$ is spin, we shall regard \eqref{eq:proj.XZ}
as the lift to $\SU(2)=\Spin(3)$ of an $\SO(3)$-(orbi)bundle of frames of $\Lambda^2_+ Z$, the bundle of self-dual 2-forms on $Z$. 

The Levi-Civita connection of $Z$ induces a connection $\eta$ on the bundle \eqref{eq:proj.XZ} which, seen as a $1$-form on $X$ with values in $\mathfrak{su}(2)$, can be written as
\begin{equation}\label{eq:eta.connection}
\eta=\eta_i \otimes T_i ,
\end{equation}
where the $T_i$ are a standard basis of $\mathfrak{su}(2)$ satisfying $[T_i,T_j]=2\epsilon_{ijk}T_k$, and the $\eta_i$ are $1$-forms on $X$. The horizontal space of the connection is $H=\ker (\eta)$. Knowing that $Z$ is anti-self-dual Einstein means that the curvature of the connection $\eta$ in \eqref{eq:eta.connection}  is given by
\begin{equation}\label{eq:Curvature_Eta}
	F_{\eta}= d\eta + \frac{1}{2}[\eta \wedge \eta] = - 2
	\omega_i \otimes T_i,
\end{equation} 
with the $2$-forms $\omega_1,\omega_2,\omega_3$ forming an orthogonal basis of $(\Lambda^2_+ H, g^{ts} \vert_{H})$ with $\vert \omega_i \vert = \sqrt{2}$. 
Notice that we have the relations
\begin{equation}\label{eq:omegas}
\omega_i\wedge\omega_j=2\delta_{ij}\pi^*\vol_Z,
\end{equation}
where $\vol_Z$ is the Riemannian volume form on $(Z,g_Z)$.

\subsection{Metrics and \texorpdfstring{G\textsubscript{2}}{G2}-structures} The $3$-Sasakian metric $g^{ts}$ can be written
\begin{equation}\label{eq:g.ts}
g^{ts} = \eta_i \otimes \eta_i +  
\pi^* g_Z 
\end{equation}
and it is well-known to be Einstein. In fact, there is a second Einstein metric on $X$ for which $\pi$ is a Riemannian submersion. This can be obtained from $g^{ts}$ by  squashing the fibers of $\pi$ by a factor which reduces the length of any $\xi_i$-orbit by $\sqrt{5}$. This yields the metric
\begin{equation}\label{eq:g.np}
	g^{np}=\frac{1}{5}\eta_i\otimes \eta_i + 
	\pi^*g_Z.
\end{equation}

 The metrics $g^{ts}$ in \eqref{eq:g.ts} and $g^{np}$ in \eqref{eq:g.np} are induced by natural distinguished $\G_2$-structures $\varphi^{ts}$ and $\varphi^{np}$ on $X^7$. (For details on $\G_2$-structures we refer the reader to \cite{G2}, for example.) Using \eqref{eq:eta.connection} and \eqref{eq:Curvature_Eta}, it is convenient to consider two  $1$-parameter families, depending on $t\in\mathbb{R}^+$ and $\varepsilon\in\{\pm 1\}$, of $\G_2$-structures on $X^7$ determined by the $3$-forms 
\begin{equation}\label{eq:G2_Structure}
	\varphi_{t,\varepsilon} =\varepsilon t^3\eta_{123}-t 
	(\eta_1\wedge \omega_1+\eta_2\wedge \omega_2+\varepsilon\eta_3\wedge \omega_3),
\end{equation}
where we have used the notation $\eta_{123}=\eta_1\wedge\eta_2\wedge\eta_3$ for brevity. Each $\varphi_{t,\varepsilon}%\varphi_{t,\nu}
$ determines the metric
\begin{equation}\label{eq:g.t.nu}
 g_{t}= t^2 \eta_i \otimes \eta_i + 
 \pi^*g_Z, 
\end{equation}
which is \emph{independent} of $\varepsilon$, and the associated $4$-forms $\psi_{t,\varepsilon}=\ast_{g_{t}} \varphi_{t,\varepsilon}$
are 
\begin{equation}\label{eq:psi.t.nu}
\psi_{t,\varepsilon} = 
\pi^*\vol_Z - t^2 
\left( \varepsilon\eta_{23} \wedge \omega_1 +  \varepsilon\eta_{31} \wedge \omega_2+\eta_{12} \wedge \omega_3 \right),
\end{equation}
 where we write $\eta_{ij}=\eta_i\wedge\eta_j$ for short.
Note that if we set
\begin{equation}\label{eq:varphi.ts.np}
\varphi^{ts}= \varphi_{1,-1}\quad\text{and}\quad \varphi^{np}=\varphi_{1/\sqrt{5},+1},
%\varphi^{ts}= \varphi_{1,1}\quad\text{and}\quad \varphi^{np}=\varphi_{1/\sqrt{5},1},
\end{equation}
then \eqref{eq:g.t.nu} shows that $\varphi^{ts}$ induces $g^{ts}$ in \eqref{eq:g.ts} and $\varphi^{np}$ induces $g^{np}$ in \eqref{eq:g.np}.

\begin{remark}
Changing the sign of the parameter $\epsilon$ corresponds to the change $\eta_3\mapsto -\eta_3$, which gives a change of orientation on the vertical space in the projection \eqref{eq:proj.XZ}.  However, since we have fixed the structure equations on $\SU(2)$, we are not free to change $\eta_3\to -\eta_3$, and so $\varepsilon$ represents a genuine parameter. 
\end{remark}

We now recall a notable class of $\G_2$-structures in this setting.

\begin{definition}\label{dfn:np}
A $\G_2$-structure determined by a 3-form $\varphi$, with dual 4-form $\psi$, is \emph{nearly parallel} if $d\varphi=\lambda\psi$ for some $\lambda\in\mathbb{R}\setminus\{0\}$.  By possibly changing orientation, we can always ensure that $\lambda>0$.
\end{definition}
 
Using the equation \eqref{eq:Curvature_Eta} for the curvature of $\eta$, we find the following structure equations:
\begin{align}
	d \eta_i & =  - 2
	\omega_i - 2 \eta_j \wedge \eta_k,\label{eq:structure.eta}\\
	d\omega_i & = 2\omega_j\wedge\eta_k-2\eta_j\wedge\omega_k,\label{eq:structure.omega}
\end{align}
for $(i,j,k)$ a cyclic permutation of $(1,2,3)$. Equations \eqref{eq:structure.eta}-\eqref{eq:structure.omega} immediately show that
\begin{equation}
 d\psi_{t,\varepsilon}=0
\end{equation}
for all $t,\varepsilon$.
Using \eqref{eq:structure.eta}--\eqref{eq:structure.omega} we compute from \eqref{eq:G2_Structure} and \eqref{eq:psi.t.nu} that the equation 
\begin{equation}\label{eq:dphipsi}
d \varphi_{t,\varepsilon} = \lambda \psi_{t,\varepsilon}
\end{equation}
only has  solutions when
\begin{equation}\label{eq:tel.np}
(t,\varepsilon,\lambda)= \left(\frac{1}{\sqrt{5}},+1,\frac{12}{\sqrt{5}}\right) \quad\text{and}\quad (t,\varepsilon,\lambda)=\left(1,-1,4\right).
\end{equation}
  By Definition \ref{dfn:np}, $(t,\varepsilon)=(1/\sqrt{5},+1)$ and $(t,\varepsilon)=(1,-1)$ are the only values for which $\varphi_{t,\varepsilon}$ 
 is  nearly parallel \cites{FKMS1997}.  We see from \eqref{eq:g.ts}, \eqref{eq:g.np} and \eqref{eq:g.t.nu}  that
 $g_{1/\sqrt{5}}=g^{np}$ and $g_1=g^{ts}$, and \eqref{eq:varphi.ts.np}, \eqref{eq:dphipsi} and \eqref{eq:tel.np} show that
 %\end{equation*}
 $\varphi^{np}$ and $\varphi^{ts}$ are nearly parallel.

We conclude this section with a familiar concrete example.

\begin{example}\label{ex:7sphere}
The standard 7-sphere $S^7$ with its round constant curvature $1$ metric $g_{S^7}$ is 3-Sasakian, and its $\SU(2)$ leaf space is $Z=S^4$, naturally endowed with its round constant curvature $4$ metric $g_Z$.  It is well-known that $S^7$ admits two homogeneous Einstein metrics:  $g_{S^7}$ (which is $g^{ts}$ in \eqref{eq:g.ts}) and its ``squashed'' metric (which is $g^{np}$ in \eqref{eq:g.np}).   Each of these Einstein metrics is induced by a homogeneous nearly parallel $\G_2$-structure, given in \eqref{eq:varphi.ts.np} by $\varphi^{ts}$ and $\varphi^{np}$ respectively.
\end{example}

\section{Deformed Hermitian-Yang--Mills connections and ASD instantons}\label{sec:Pullback}

In this section we provide examples of deformed $\G_2$-instantons arising from lower-dimensional geometries: specifically, deformed Hermitian-Yang--Mills connections on Calabi--Yau 3-folds and anti-self-dual instantons on anti-self-dual Einstein 4-orbifolds and hypersymplectic 4-manifolds. 

\subsection{Deformed Hermitian-Yang--Mills connections} We recall the following definition, which originated in \cites{LYZ,MMMS}.

\begin{definition}
Let $(Y,\omega)$ be a K\"ahler $n$-fold, where $\omega$ is the K\"ahler form, and let $L$ be a Hermitian complex line bundle on $Y$.  A unitary connection $A$ on $L$ is a \emph{deformed Hermitian-Yang--Mills connection} (with phase $e^{i\alpha}$) if 
\begin{equation}\label{eq:dHYM.1}
F_A^{(0,2)}=0\quad\text{and}\quad \Im(e^{-i\alpha}(\omega+F_A)^n)=0.
\end{equation}
When $Y$ is a Calabi--Yau manifold, deformed Hermitian-Yang--Mills connections  are, in a sense, ``mirror'' to special Lagrangian $n$-folds.
\end{definition}

We are interested in the case where $(Y,\omega,\Omega)$ is a Calabi--Yau 3-fold, with holomorphic volume form $\Omega$, and $A$ is a deformed Hermitian-Yang--Mills connection with phase $1$.  In this case \eqref{eq:dHYM.1} can be rewritten as
\begin{equation}\label{eq:dHYM.2}
F_A\wedge\Im\Omega=0\quad\text{and}\quad \frac{1}{6}F_A^3+F_A\wedge \frac{1}{2}\omega^2=0.
\end{equation}
The analogy between \eqref{eq:dG2.intro} and \eqref{eq:dHYM.2} should be clear.
We then provide an observation which extends Lemma 5.5 in \cites{KawaiYamamoto}.

\begin{lemma}\label{lem:CY3}
Let $(Y,\omega,\Omega)$ be a Calabi--Yau $3$-fold and let $L$ be a Hermitian complex line bundle on $Y$.  Let $\pi:X^7\to Y$ be an $S^1$-bundle over $Y$ with a connection 1-form $\eta$, which is Hermitian-Yang--Mills, endowed with the standard $\G_2$-structure
\begin{equation}\label{eq:product.G2}
\varphi=\eta\wedge\pi^*\omega+\pi^*\Re\Omega\quad\text{and}\quad\psi=\frac{1}{2}\pi^*\omega^2-\eta\wedge\pi^*\Im\Omega.
\end{equation}  
Note that as $\eta$ is Hermitian-Yang--Mills we have $d\psi=0$.

Then $A$ is a deformed Hermitian-Yang--Mills connection with phase $1$ on $L$ if and only if $\pi^*A$ is a deformed $\G_2$-instanton on $\pi^*L$. 
\end{lemma}

\begin{proof}
The proof follows immediately from \eqref{eq:dG2.intro}, \eqref{eq:dHYM.2} and \eqref{eq:product.G2}, just as in the proof of Lemma 5.5 in \cites{KawaiYamamoto}.
\end{proof}

\begin{remark}
Lemma \ref{lem:CY3} has a well-known analogue where deformed Hermitian-Yang--Mills connections and deformed $\G_2$-instantons are replaced by Hermitian-Yang--Mills conections and $\G_2$-instantons.  When $X^7=S^1\times Y$ with the product $\G_2$-structure and $\eta=d\theta$  for $\theta$ the coordinate on $S^1$, $\G_2$-instantons on $\pi^*L$ are related via a ``broken gauge'' with the pullback of Hermitian-Yang--Mills connections over $Y$ \cites{YWang}. In particular, this implies that $\pi^*A$ is a $\rG_2$-instanton if and only if $A$ is a Hermitian--Yang--Mills connection.
\end{remark}

There are now many examples of deformed Hermitian-Yang--Mills connections, in particular provided by the recent relationship between existence of such connections and stability proved in \cites{GaoChen}.  We may construct a simple example of a $\G_2$-manifold, i.e.~$X^7$ with a torsion-free $\G_2$-structure $\varphi$ (satisfying $d\varphi=0$ and $d\psi=0$), with a non-trivial line bundle admitting a deformed $\G_2$-instanton which is \emph{not} a $\G_2$-instanton as follows.

\begin{example}\label{ex:product}  Suppose that $(Y,\omega,\Omega)$ is a Calabi--Yau 3-fold such that $[\sqrt{3}\omega]$ is an integral class in $H^2(Y)$, and so defines a Hermitian complex line bundle $L$ with a unitary connection $A$ such that $F_A=i\sqrt{3}\omega$.  Then \eqref{eq:dHYM.2} is satisfied and so $A$ is a deformed Hermitian-Yang--Mills connection with phase $1$.

If we let $X=S^1\times Y$ with the product torsion-free $\G_2$-structure as in \eqref{eq:product.G2}, where $\pi:X\to Y$ is the natural projection and $\eta=d\theta$ for $\theta$ the coordinate on $S^1$, then $\pi^*A$ is a deformed $\G_2$-instanton on $\pi^*L$ by Lemma \ref{lem:CY3}.  Notice that 
$\pi^*F_A^3\neq 0$ and so $\pi^*A$ is \emph{not} a $\G_2$-instanton on $X=S^1\times Y$. 
\end{example}

\begin{remark} One can perform a similar construction to Example \ref{ex:product} for so-called Calabi--Yau links $X^7$ in $S^9$, which are nontrivial $S^1$-bundles over Calabi--Yau 3-(orbi)folds $Y$ arising as hypersurfaces in $\mathbb{CP}^4$, to obtain deformed $\G_2$-instantons which are not $\G_2$-instantons on $X$.  The study of $\G_2$-instantons on Calabi--Yau links was initiated in \cites{CYlinks}, using Hermitian-Yang--Mills connections on $Y$.
\end{remark}

\subsection{ASD instantons on anti-self-dual Einstein 4-orbifolds}\label{ss:ASD.1}

Let $X^7$ be a 3-Sasakian 7-manifold as in Definition \ref{dfn:3Sasakian} and let $(Z^4,g_Z)$ as in \eqref{eq:proj.XZ} be the $\SU(2)$ leaf space.  Recall that $(Z^4,g_Z)$ is an anti-self-dual Einstein 4-orbifold, and recall the $\G_2$-structures on $X$ whose 4-forms are given by $\psi_{t,\varepsilon}$ in \eqref{eq:psi.t.nu}.   In particular, recall the forms $\omega_i$ in \eqref{eq:Curvature_Eta}, which are pullbacks of self-dual 2-forms on $Z$, used to construct $\psi_{t,\varepsilon}$.

We now have the following simple observation concerning \emph{anti-self-dual} (ASD) instantons on $Z$, i.e.~connections $A$ on $Z$ whose curvature satisfies
\begin{equation}\label{eq:ASD}
 F_A=-*F_A.
\end{equation}
Since $\pi^*F_A^3=0$ automatically for dimension reasons, for any connection $A$ on $Z$, we see from \eqref{eq:dG2.intro} that the notions of deformed $\G_2$-instanton and $\G_2$-instanton coincide for $\pi^*A$.  We may thus obtain trivial examples of deformed $\G_2$-instantons as follows.

\begin{lemma}\label{lem:ASD.1} Let $X^7$ be a $3$-Sasakian $7$-manifold and let $Z$ be its $\SU(2)$ leaf space as in \eqref{eq:proj.XZ}.  Let $L$ be a Hermitian complex line bundle on $Z$, and let $A$ be a unitary connection on $L$.  

Then  $\pi^*A$ is a (deformed) $\G_2$-instanton on $\pi^*L$ over $X$, with respect to some (and hence all) $\psi_{t,\varepsilon}$ in \eqref{eq:psi.t.nu} if and only if $A$ is an ASD instanton.
\end{lemma}

\begin{proof}
First, let $A$ be an anti-self-dual (ASD) instanton on $Z$.
Then, since $F_A$ is anti-self-dual and the forms $\omega_i$ appearing in \eqref{eq:psi.t.nu} are self-dual on the horizontal space $H$ for the projection \eqref{eq:proj.XZ}, we see that $\pi^*F_A\wedge\psi_{t,\varepsilon}=0$, i.e.~that $\pi^*A$ is a $\G_2$-instanton. 

 Conversely, if $\pi^*F_A\wedge\psi_{t,\varepsilon}=0$ for some $t$ and $\varepsilon$, then we must have
 \begin{equation}\label{eq:wedge}
  \pi^*F_A\wedge\omega_i=0\quad\text{for all $i$.}
 \end{equation}
Since the $\omega_i$ span the self-dual 2-forms on $H$, \eqref{eq:wedge} implies that \eqref{eq:ASD} holds, i.e.~$A$ is an ASD instanton.
\end{proof}

\begin{remark}
 Lemma \ref{lem:ASD.1} has some overlap with Proposition 18 in \cites{BallOliveira}.
\end{remark}

We now give an example of using Lemma \ref{lem:ASD.1} which will be useful later.

\begin{example}\label{ex:AW}
  Let $X^7$ be the Aloff--Wallach space $\SU(3)/\U(1)_{1,1}$ where 
  \begin{equation}\label{eq:AW}
   \U(1)_{1,1}=\left\{\left(\begin{array}{ccc} e^{i\theta} & 0 & 0\\
                        0 & e^{i\theta} & 0\\
                        0 & 0 & e^{-2i\theta}
                            \end{array}
\right)\in\SU(3)\,:\theta\in\mathbb{R}\right\}.
  \end{equation}
Then $X$ is a homogeneous 3-Sasakian 7-manifold whose $\SU(2)$ leaf space is $Z=\overline{\mathbb{CP}^2}$.  Let $L=\mathcal{O}(k)$ on $Z$  for $k\in\mathbb{Z}$.  The  connection $A$ on $L$ with harmonic curvature will be unitary and have the property that $F_A$ is a multiple of the Fubini--Study form, and so will be an ASD instanton (since $\overline{\mathbb{CP}^2}$ has the opposite orientation).  Moreover, $A$ will be non-trivial whenever $k\neq 0$.  Lemma \ref{lem:ASD.1} then gives a deformed $\G_2$-instanton (which is a non-trivial $\G_2$-instanton for $k\neq 0$) on $\pi^*L$ with respect to both of the homogeneous nearly parallel $\G_2$-structures $\varphi^{ts}$ and $\varphi^{np}$ in \eqref{eq:varphi.ts.np} on $X$.  
\end{example}

\begin{remark}
Gauge theory on $X=\SU(3)/\U(1)_{1,1}$, in particular $\G_2$-instantons with respect to the two homogeneous nearly parallel $\G_2$-structures, is studied in some detail in \cites{BallOliveira}.
\end{remark}

\subsection{ASD instantons on hypersymplectic 4-manifolds}

We shall now give some simple examples arising from pull-backs of anti-self-dual connections on a hypersymplectic $4$-dimensional manifold.  As we shall see, this is directly analogous to the construction arising from anti-self-dual Einstein 4-orbifolds considered in the previous subsection. 

\begin{definition}\label{dfn:hyp}
A \emph{hypersymplectic} structure on an oriented $4$-manifold $Z^4$ is a triple of closed $2$-forms $(\omega_1,\omega_2,\omega_3)$ on $Z$ so that, for a volume form $\vol_Z$ on $Z$,
\begin{equation}\label{eq:hyp}
 \omega_i \wedge \omega_j = 2Q_{ij} \vol_Z,
\end{equation}
where the matrix $Q_{ij}$ is positive definite.  We may choose $\vol_Z$ in \eqref{eq:hyp} so that the matrix $Q_{ij}$ has determinant one.  In this manner, the hypersymplectic structure determines a Riemannian metric $g_Z$ on $Z$ whose bundle of self-dual two forms is spanned by $\lbrace \omega_1 , \omega_2 , \omega_3 \rbrace$ and whose Riemannian volume form is $\vol_Z$.
\end{definition}

\begin{example} If $(Z^4,g_Z)$ is a K3 surface with $(\omega_1,\omega_2,\omega_3)$ a hyperk\"ahler triple, then \eqref{eq:hyp} is satisfied with $Q_{ij}=\delta_{ij}$ and $g_Z$ in Definition \ref{dfn:hyp} is the associated Ricci-flat K\"ahler metric. See \cites{Donaldson2006} for more about hypersympletic structures and their relation to kyperk\"ahler geometry.
\end{example}

Consider a $3$-torus bundle $\pi:X \to Z$ over such a hypersymplectic $4$-manifold $Z$ with an anti-self-dual connection $\eta$. We regard the connection $\eta =(\eta_1,\eta_2,\eta_3) \in \Omega^1(X, \mathbb{R}^3)$ as three $1$-forms on the total space whose curvatures $d \eta_i$ are the pullback of anti-self-dual $2$-forms on $(Z,g_Z)$. Then, we consider $\G_2$-structures on $X$ whose corresponding $4$-forms are
\begin{equation}\label{eq:hyp.psi}
 \psi_t= \pi^*\vol_Z - t^2 \left( \eta_{23} \wedge \pi^*\omega_1 + \eta_{31} \wedge \pi^*\omega_2 + \eta_{12} \wedge \pi^*\omega_3 \right) ,
\end{equation}
for some constant $t>0$. (For more details on the relation between $\G_2$-structures and hypersympletic structures see \cites
{Fine2018}.) We see that $d\psi_t=0$ as the $d\eta_i$ are assumed to be the pull-back of anti-self-dual $2$-forms and thus $d\eta_i \wedge \omega_j=0$ for all $i,j$.  This is the only point for which the condition that $\eta$ is anti-self-dual is used.

We now construct some simple deformed $\G_2$-instantons with respect to $\psi_t$ in \eqref{eq:hyp.psi}. We regard these as being trivial examples since, as in Lemma \ref{lem:ASD.1} above, they are $\G_2$-instantons for which the cubic term in the curvature vanishes.  The proof is almost identical to Lemma \ref{lem:ASD.1} so we omit it.

\begin{lemma}\label{lem:ASD.2}
Let $Z$ be a hypersymplectic $4$-manifold, let $L$ be a Hermitian complex line bundle on $Z$ and let $A$ be a unitary connection on $L$.   Let $\pi:X\to Z$ be a $T^3$-bundle over $Z$ with anti-self-dual connection $\eta$ and let $\psi_t$ be as in \eqref{eq:hyp.psi}.  

Then $\pi^*A$ is a (deformed) $\G_2$-instanton on $\pi^*L$ with respect to some (and hence all) $\psi_t$ if and only if $A$ is an ASD instanton.
\end{lemma}

\begin{example}\label{ex:K3}
	Let $Z$ be a K3 surface and $\omega_1,\omega_2,\omega_3$ a hyperk\"ahler triple, and let $X=T^3 \times Z$ so that $\eta$ is the trivial flat connection on $\pi:X\to Z$. Note that in this case $X$ is a $\G_2$-manifold with the product $\G_2$-structure whose 4-form is $\psi=\psi_1$ in \eqref{eq:hyp.psi}.
	
	Denote by $\cH_+= \langle \omega_1, \omega_2,\omega_3 \rangle$ the space of self-dual harmonic $2$-forms and by $\cH_-$ that of anti-self-dual ones. Then $H^2(X,\mathbb{R}) \cong \cH_+ \oplus \cH_-$ and suppose that $\cH_- \cap H^2(X, \mathbb{Z}) \neq 0$. Then, for any complex line bundle so that $c_1(L) \in \cH_- \cap H^2(X, \mathbb{Z})$, Lemma \ref{lem:ASD.2} applies and we obtain a connection $A$ which is both a $\G_2$-instanton and a deformed $\G_2$-instanton for $\psi$ on the $\G_2$-manifold $X=T^3\times Z$.
\end{example}
 
\begin{remark} Similar constructions of $\G_2$-instantons to Example \ref{ex:K3}, also performed in the higher rank case, can be found in \cites{Clarke2020}.
\end{remark}

\section{Examples}\label{sec:Examples}

We now turn to the main goal of this article, which is to construct the first non-trivial examples of deformed $\G_2$-instantons.  At the end of the section, we shall also discuss implications of this construction for the deformation theory of deformed $\G_2$-instantons.

In this section, unless we state otherwise, we let $X^7$ be a 3-Sasakian 7-manifold as in Definition \ref{dfn:3Sasakian}, and we shall use the notation introduced in section \ref{sec:3Sasakian}.  In particular we let $Z^4$ be the leaf space of the $\SU(2)$ action on $X$ given in \eqref{eq:proj.XZ}. We also let $L_0$ denote the trivial complex line bundle over $X$.

\subsection{\texorpdfstring{G\textsubscript{2}}{G2}-instantons}\label{ss:Examples_TRivial_Line_Bundle}

We consider connections $A$ on the trivial complex line bundle 
$L_0$
 over $X$, given by
\begin{equation}\label{eq:Connection}
A=i(a_1 \eta_1 +a_2 \eta_2+a_3 \eta_3),
\end{equation}
for $a_1,a_2,a_3 \in \mathbb{R}$, where the $\eta_i$ are given in \eqref{eq:eta.connection}.  (Here, we identify the Lie algebra $\mathfrak{u}(1)$ with $i\mathbb{R}$.) The curvature of $A$ is given by
\begin{equation}\label{eq:FA} F_A= - 2ia_1 \left(
\omega_1 +  \eta_{23} \right) - 2ia_2 \left(  
\omega_2 +  \eta_{31} \right) - 2ia_3 \left(  
\omega_3 +  \eta_{12}\right),
 \end{equation}
where the $\omega_i$ are given in \eqref{eq:Curvature_Eta}, and we recall that $\eta_{ij}=\eta_i\wedge\eta_j$.  Using \eqref{eq:omegas}, \eqref{eq:psi.t.nu} and \eqref{eq:FA},  we compute
\begin{align}\label{eq:F_Wedge_Psi}
	F_A \wedge \psi_{t,\varepsilon} 
		= - 2i \big((1-2\varepsilon t^2) \left( a_1 \eta_{23} + a_2 \eta_{31}\right) + (1-2t^2)a_3 \eta_{12} \big) \wedge \pi^*\vol_Z.
\end{align}
From \eqref{eq:G2inst.intro} and \eqref{eq:F_Wedge_Psi} we are led to the following conclusion.

\begin{proposition}\label{prop:G2_Inst_Trivial}
	Suppose that $A$ in \eqref{eq:Connection} is a $\G_2$-instanton with respect to $\psi_{t,\varepsilon}$ in \eqref{eq:psi.t.nu}. Then  
	\begin{itemize}
		\item $a_1=a_2=a_3=0$, in which case $A$ is the trivial flat connection, or
		\item $t=1/\sqrt{2}$, $\varepsilon=-1$, $a_1=a_2=0$ and $a_3\neq 0$, so $A=ia_3\eta_3$ is a $\G_2$-instanton with respect to $\psi_{1/\sqrt{2},-1}$, or
		\item $t=1/\sqrt{2}$ and $\varepsilon=+1$,   
		in which case all $A$ in \eqref{eq:Connection} are $\G_2$-instantons with respect to $\psi_{1/\sqrt{2},+1}$.
			\end{itemize}
\end{proposition}

\begin{remark}
	The $\G_2$-structures given by $(t,\varepsilon)=(1/\sqrt{2},+1)$ are special, as they support a real $3$-parameter family of $\G_2$-instantons on the trivial complex line bundle $L_0$ on $X$ by Proposition \ref{prop:G2_Inst_Trivial}. This extends observations made in Proposition 17 of \cites{BallOliveira}.
	Moreover, the $\G_2$-structures given by $(t,\varepsilon)=(1/\sqrt{2},-1)$ admit a real $1$-parameter family of $\G_2$-instantons on $L_0$ on $X$.
\end{remark}

\begin{remark} There are examples of $\G_2$-instantons on higher rank bundles on $X$, such as
 Example 2 in \cites{Clarke2020B}, which describes an irreducible $\G_2$-instanton with gauge group $\SO(3)$ in this setting.
\end{remark}

\subsection{Deformed \texorpdfstring{G\textsubscript{2}}{G2}-instantons}

We now analyze solutions to the deformed $\G_2$-instanton (or, equivalently, deformed Donaldson--Thomas connection) equation with respect to $\psi_{t,\varepsilon}$:
\begin{equation}\label{eq:dG2}
%c 
\frac{1}{6}F_A^3+F_A \wedge \psi_{t,\varepsilon} 
=0,
\end{equation}
for $A$ a connection on the trivial complex line bundle $L_0$ on $X$. For $A$ as in \eqref{eq:Connection} we compute
\begin{equation}\label{eq:Cubic_Term}
\frac{1}{6}F_A^3=
8i \left( a_1^2+a_2^2+a_3^2 \right) \left( a_1 \eta_{23} + a_2 \eta_{31} + a_3 \eta_{12} \right) \wedge \pi^*\vol_Z,
\end{equation}
Using \eqref{eq:F_Wedge_Psi} and \eqref{eq:Cubic_Term}, we see that $A$ solves \eqref{eq:dG2} if and only if 
\begin{equation}\label{eq:dG2.L0}
\big(4(a_1^2+a_2^2+a_3^2)-(1-2\varepsilon t^2)\big)(a_1\eta_{23}+a_2\eta_{31})+\big(4(a_1^2+a_2^2+a_3^2)-(1-2t^2)\big)a_3\eta_{12}=0.
\end{equation}
We see that \eqref{eq:dG2.L0} always has the solution $a_1=a_2=a_3=0$, which corresponds to the flat connection.  Otherwise, if $\varepsilon=+1$ then we must have
\begin{equation}\label{eq:constraint}
  a_1^2+a_2^2+a_3^2  = \frac{1}{4} \left(1- 2t^2 \right) .
  \end{equation}
We immediately see that \eqref{eq:constraint} can be solved for a non-flat connection $A$ if and only if $2t^2 < 1$, in which case there is a whole $2$-sphere of solutions.  If, instead, $\varepsilon=-1$ then if $a_1^2+a_2^2\neq 0$ we must have
\begin{equation}\label{eq:constraint.2}
 a_1^2+a_2^2=\frac{1}{4}(1+2t^2)\quad\text{and}\quad a_3=0,
\end{equation}
and if $a_3\neq 0$ we must have
\begin{equation}\label{eq:constraint.3}
 a_1^2+a_2^2=0\quad\text{and}\quad a_3^2=\frac{1}{4}(1-2t^2).
\end{equation}
Clearly, \eqref{eq:constraint.2} gives a circle of solutions for any $t>0$, whereas \eqref{eq:constraint.3} gives non-trivial solutions if and only if $2t^2<1$, in which case there are two solutions.

We state these findings as follows.

\begin{proposition} \label{prop:examples}  Let $a_1,a_2,a_3\in\mathbb{R}$ and let $A=ia_j\eta_j$ as in \eqref{eq:Connection} be a connection on the trivial complex line bundle $L_0$ on $X$.  Then $A$ is a deformed $\G_2$-instanton with respect to $\psi_{t,\varepsilon}$ in \eqref{eq:psi.t.nu} if and only if either $a_1=a_2=a_3=0$, so $A$ is the trivial flat connection, or one of the following holds:
\begin{itemize}
\item $t\in(0,1/\sqrt{2})$, $\varepsilon=+1$ and $a_1,a_2,a_3$ satisfy \eqref{eq:constraint} (so there is a $2$-sphere of solutions);
\item $t\in(0,1/\sqrt{2})$, $\varepsilon=-1$ and $a_1,a_2,a_3$ satisfy \eqref{eq:constraint.2} or \eqref{eq:constraint.3} (so the solutions consist of a circle and two points); 
\item $t\geq 1/\sqrt{2}$, $\varepsilon=-1$, and $a_1,a_2,a_3$ satisfy \eqref{eq:constraint.2} (so there is a circle of solutions).
\end{itemize}
\end{proposition}

By \eqref{eq:varphi.ts.np}, Proposition \ref{prop:examples} immediately gives the following two results.

\begin{corollary}\label{cor:examples.np}  Recall the two nearly parallel $\G_2$-structures $\varphi^{np}$ and $\varphi^{ts}$ on $X$ given in \eqref{eq:varphi.ts.np}.
\begin{itemize}
 \item There is a $2$-sphere of non-trivial deformed $\G_2$-instantons on $L_0$ over $X$ with respect to $\varphi^{np}$ arising from \eqref{eq:Connection}.
 \item  There is a circle of non-trivial deformed $\G_2$-instantons on $L_0$ over $X$ with respect to $\varphi^{ts}$ arising from \eqref{eq:Connection}.
\end{itemize}
\end{corollary}

\begin{remark}
 Corollary \ref{cor:examples.np} demonstrates how Proposition \ref{prop:examples} can be used to show that deformed $\G_2$-instantons can discriminate between $\G_2$-structures on $X$; in particular, between the two natural nearly parallel $\G_2$-structures on $X$.  We also see that the family of deformed $\G_2$-structures for these two nearly parallel $\G_2$-structures has different dimensions, and there is no obvious relation between them.
\end{remark}

\begin{corollary}\label{cor:examples.metrics}
Recall the two Einstein metrics $g^{ts}$ and $g^{np}$ on $X$ given in \eqref{eq:g.ts} and \eqref{eq:g.np}, and recall that $\varphi_{1,\varepsilon}$ induces $g^{ts}$ and $\varphi_{1/\sqrt{5},\varepsilon}$ induces $g^{np}$ for $\varepsilon\in\{\pm1\}$.
\begin{itemize}
 \item Using the ansatz \eqref{eq:Connection}, there is a circle of non-trivial deformed $\G_2$-instantons with respect to $\varphi_{1,-1}$, whereas there are no non-trivial deformed $\G_2$-instantons with respect to $\varphi_{1,+1}$.
 \item Using the ansatz \eqref{eq:Connection}, there is a circle plus two isolated examples of non-trivial deformed $\G_2$-instantons with respect to $\varphi_{1/\sqrt{5},-1}$, whereas there is a $2$-sphere of non-trivial deformed $\G_2$-instantons with respect to $\varphi_{1/\sqrt{5},+1}$.
\end{itemize}
\end{corollary}

\begin{remark}
 Corollary \ref{cor:examples.metrics} indicates how Proposition \ref{prop:examples} shows that deformed $\G_2$-instantons can be used to distinguish between isometric $\G_2$-structures on $X$; in particular, between the two natural Einstein metrics on $X$. However, we observe that for these two Einstein metrics, whilst the spaces of deformed $\G_2$-instantons having the form \eqref{eq:Connection} are very different for the two isometric $\G_2$-structures, their Euler characteristics are the same. This is pertinent since one might hope to use the Euler characteristic of the moduli space as a possible enumerative invariant for deformed $\G_2$-instantons.
\end{remark}

We give a concrete example of the construction.

\begin{example}\label{ex:7sphere.dG2}
Take the 7-sphere $S^7$ as in Example \ref{ex:7sphere} and let $L_0$ be the trivial complex line bundle over $S^7$.   Corollary \ref{cor:examples.np} gives a $2$-sphere of deformed $\G_2$-instantons on $L_0$ over $(S^7,\varphi^{np})$, and a circle of deformed $\G_2$-instanton on $L_0$ over $(S^7,\varphi^{ts})$.  On the other hand, Corollary \ref{cor:examples.metrics} shows that we have a family of deformed $\G_2$-instantons on $L_0$ consisting of a circle plus two further points for another $\G_2$-structure inducing the squashed metric $g^{np}$, and we have no known non-trivial deformed $\G_2$-instantons on $L_0$ for another $\G_2$-structure inducing the round metric $g^{ts}$.

In this way, deformed $\G_2$-instantons on the trivial complex line bundle can be used to distinguish between the two homogeneous nearly parallel $\G_2$-structures on $S^7$, and between isometric $\G_2$-structures for the two homogeneous Einstein metrics on $S^7$.
\end{example}

\subsection{Non-trivial line bundles}
One may ask whether there are non-trivial examples of deformed $\G_2$-instantons on non-trivial line bundles.  The natural approach is to use a non-trivial bundle $L$ over $Z$ equipped with an anti-self-dual connection $A_0$, using the ideas in subsection \ref{ss:ASD.1}.  We shall give a particular example, which may clearly be generalized to  other 3-Sasakian 7-manifolds,  of a existence result for non-trivial deformed $\G_2$-instantons on non-trivial line bundles.

Consider the setting of Example \ref{ex:AW}, where $X$ is the Aloff--Wallach space $\SU(3)/\U(1)_{1,1}$, where $\U(1)_{1,1}$ is given in \eqref{eq:AW}, and recall that $Z=\overline{\mathbb{CP}^2}$. Let $L=\cO_{\overline{\mathbb{CP}^2}}(k)$ for some $k \in \mathbb{Z}$, and let $A_0$ be the connection on $L$ with harmonic curvature (which must be an ASD instanton as observed in Example \ref{ex:AW}).

Recalling the $\eta_i$ in \eqref{eq:eta.connection}, we may consider a connection $A$ on $\pi^*L$ over $X$ given by
\begin{equation}\label{eq:connection.AW}
A=\pi^*A_0 + ia , \ \text{for} \ \ a=  a_j \eta_j,
\end{equation}
where $a_1,a_2,a_3\in\mathbb{R}$.  The curvature of $A$ is $F_A=F_{A_0} + ida$ and as $A_0$ is anti-self-dual we have $\pi^*F_{A_0} \wedge \psi_{t,\varepsilon}=0$ by Lemma \ref{lem:ASD.1}. Hence,
\begin{equation}
F_A \wedge \psi_{t,\varepsilon}=ida \wedge \psi_{t,\varepsilon}=- 2i\big( (1-2\varepsilon t^2) ( a_1 \eta_{23} + a_2 \eta_{31})+(1-2t^2)a_3 \eta_{12} \big) \wedge \pi^*\vol_Z.
\end{equation}
by \eqref{eq:F_Wedge_Psi}. As for the cubic term in the curvature we find that 
\begin{equation}\label{eq:FA3.1}
F_A^3=3i\pi^*F_{A_0}^2 \wedge da - 3 \pi^*F_{A_0} \wedge (da)^2 -i(da)^3
\end{equation} 
as $F_{A_0}^3=0$ for dimensional reasons. By inspection, we see that $(da)^2 = \beta_i \wedge \omega_i $ for some $2$-forms $\beta_i$. As the $\omega_i$ are self-dual and $F_{A_0}$ is anti-self-dual we find that  $F_{A_0} \wedge (da)^2=0$.  Hence, \eqref{eq:FA3.1} becomes
\begin{equation}\label{eq:FA3.2}
F_A^3 = -i(da)^3 + 3i\pi^*F_{A_0}^2 \wedge da.
\end{equation}
The first term in \eqref{eq:FA3.2} is given by the right-hand side of \eqref{eq:Cubic_Term} (multiplied by $6$). As for the second term, we find that $F_{A_0}^2 = |F_{A_0}|^2 \vol_Z$, with $|\cdot|$ denoting the norm with respect to $g_Z$, since $F_{A_0}$ is anti-self-dual (recalling that $F_{A_0}$ is imaginary-valued).  Thus, using \eqref{eq:FA}, we see that
\begin{equation}\label{eq:FA3.3}
3i\pi^*F_{A_0}^2 \wedge da = -6i \pi^*|F_{A_0}|^2 \left( a_1 \eta_{23}  + a_2 \eta_{31} +a_3 \eta_{12}\right) \wedge \pi^*\vol_Z.
\end{equation}
Inserting \eqref{eq:Cubic_Term} and \eqref{eq:FA3.3} in \eqref{eq:FA3.2} shows that the deformed $\G_2$-instanton equation  \eqref{eq:dG2} for $\psi_{t,\varepsilon}$ is equivalent to
\begin{multline}\label{eq:dG2.Ok}
\big(8(a_1^2+a_2^2+a_3^2)-\pi^*|F_{A_0}|^2-2(1-2\varepsilon t^2)\big) (a_1\eta_{23}+a_2\eta_{31})\\ +\big(8(a_1^2+a_2^2+a_3^2)-\pi^*|F_{A_0}|^2-2(1-2t^2)\big)a_3\eta_{12}=0. 
\end{multline}
At this point, we use the fact that $F_{A_0}=ik\omega$, where $\omega$ is the Fubini--Study form on $\overline{\mathbb{CP}^2}$, and thus  $|F_{A_0}|^2=2k^2$.  (Note that we need $|F_{A_0}|^2$ to be constant in order for \eqref{eq:dG2.Ok} to have a solution for constant $a_1,a_2,a_3$ which are not all zero.)
Inserting this into \eqref{eq:dG2.Ok} for $\varepsilon=+1$ gives that non-trivial solutions must satisfy
\begin{equation}\label{eq:Solution_Intermediate_2}
a_1^2+a_2^2+a_3^2  = \frac{1}{4}(1-2t^2+k^2).
%\nu^4 \left( \frac{4 \pi^2 \langle c_1(L)^2, [Z] \rangle}{8 \left(\frac{s}{48}\right)^2 \nu^4 \mathrm{Vol}(Z)} + \frac{1}{24c} \left( \frac{2t^2}{\nu^2} - 1 \right) \right).
\end{equation} 
We see that \eqref{eq:Solution_Intermediate_2} has non-trivial solutions for $a_1,a_2,a_3$ if and only if $2t^2<1+k^2$.  We therefore obtain  deformed $\G_2$-instantons with respect to $\psi_{t,+1}$ on $\pi^*\mathcal{O}(k)$ for $k\neq 0$ arising from the ansatz \eqref{eq:connection.AW} which are \emph{not} given by the pullback of the ASD instanton on $\mathcal{O}(k)$ (which is a deformed $\G_2$-instanton by Lemma \ref{lem:ASD.1}) for these values of $t$.

If we instead look at \eqref{eq:dG2.Ok} for $\epsilon=-1$ then for non-trivial solutions we either have
\begin{gather}\label{eq:Solution_Intermediate_3}
 a_1^2+a_2^2=\frac{1}{4}(1+2t^2+k^2)\quad\text{and}\quad a_3=0,\;\,\text{or}\\
a_1^2+a_2^2=0\quad\text{and}\quad a_3^2=\frac{1}{4}(1-2t^2+k^2).\label{eq:Solution_Intermediate_4}
\end{gather}
We see that \eqref{eq:Solution_Intermediate_3} admits non-trivial solutions for all $t$, whereas \eqref{eq:Solution_Intermediate_4} admits non-trivial solutions if and only if $2t^2<1+k^2$, just as for \eqref{eq:Solution_Intermediate_2}.

Overall, we have the following proposition, which generalizes Proposition \ref{prop:examples} for the case of the Aloff--Wallach space $\SU(3)/\U(1)_{1,1}$.

\begin{proposition}\label{prop:AW} Let $X$ be the Aloff--Wallach space $\SU(3)/\U(1)_{1,1}$ as in Example \ref{ex:AW}, with projection $\pi:X\to Z=\overline{\mathbb{CP}^2}$. 
	Let $L=\cO(k)$ on $Z$ for $k\in\mathbb{Z}$, let $A_0$ be the unitary ASD instanton on $L$ and let $A$ be a connection on $\pi^*L$ as in \eqref{eq:connection.AW} which is a deformed $\G_2$-instanton with respect to $\psi_{t,\varepsilon}$ given in \eqref{eq:psi.t.nu}.  Then either $a_1=a_2=a_3=0$, and so $A=\pi^*A_0$ (and thus is a $\G_2$-instanton), or one of the following holds:
	\begin{itemize}
	\item $t\in(0,\sqrt{(1+k^2)/2})$, $\varepsilon=+1$ and  $a_1,a_2,a_3$ satisfy \eqref{eq:Solution_Intermediate_2};
	\item $t\in(0,\sqrt{(1+k^2)/2})$, $\varepsilon=-1$ and  $a_1,a_2,a_3$ satisfy \eqref{eq:Solution_Intermediate_3} or \eqref{eq:Solution_Intermediate_4};
	\item $t\geq \sqrt{(1+k^2)/2}$, $\varepsilon=-1$ and $a_1,a_2,a_3$ satisfy \eqref{eq:Solution_Intermediate_3}.
	\end{itemize}
\end{proposition}

\begin{remark} We see that \eqref{eq:Solution_Intermediate_4} has non-trivial solutions for $t=1$ if and only if $|k|>1$.  In particular, 
 Proposition \ref{prop:AW} gives non-trivial deformed $\G_2$-instantons on $\pi^*\mathcal{O}(k)$ over $X$ with respect to the $\G_2$-structure $\varphi_{1,+1}$, which induces the 3-Sasakian metric $g^{ts}$, if and only if $ |k|>1$.  Moreover, when $|k|>1$, Proposition \ref{prop:AW} gives a 2-sphere of non-trivial deformed $\G_2$-instantons on $\pi^*\mathcal{O}(k)$ with respect to $\varphi_{1,+1}$, whereas it gives a family of deformed $\G_2$-instantons consisting of a circle and two points with respect to $\varphi^{ts}$.  
\end{remark}

Proposition \ref{prop:AW} has the following immediate corollary.

\begin{corollary}\label{cor:AW}
 Let $X$ be the Aloff--Wallach space $\SU(3)/\U(1)_{1,1}$ as in Example \ref{ex:AW}, with projection $\pi:X\to \overline{\mathbb{CP}^2}$, and recall the nearly parallel $\G_2$-structures $\varphi^{np}$ and $\varphi^{ts}$ on $X$ given in \eqref{eq:varphi.ts.np}.
 \begin{itemize}
  \item For every $k\in\mathbb{Z}$, there is a $2$-sphere of non-trivial deformed $\G_2$-instantons on $\pi^*\mathcal{O}(k)$ with respect to  $\varphi^{np}$. 
  \item If $k\in\{0,\pm 1\}$, there is a circle of non-trivial deformed $\G_2$-instantons on $\pi^*\mathcal{O}(k)$  with respect to $\varphi^{ts}$, and if $k\in\mathbb{Z}\setminus\{0,\pm1\}$ there is a circle and two further examples of non-trivial deformed $\G_2$-instantons on $\pi^*\mathcal{O}(k)$  with respect to $\varphi^{ts}$.
 \end{itemize}
\end{corollary}

\begin{remark}
 Proposition \ref{prop:AW} continues to demonstrate how deformed $\G_2$-instantons can distinguish between nearly parallel $\G_2$-structures and isometric $\G_2$-structures. Moreover, for isometric $\G_2$ structures, the observed equality between the Euler characteristics of the families of deformed $\G_2$-instantons, having the form \eqref{eq:Connection}, continues to hold in this setting.
\end{remark}

\subsection{Moduli spaces} We now make some observations on the families of deformed $\G_2$-instantons we have constructed, and their relation to the deformation theory of deformed $\G_2$-instantons developed in \cite{KawaiYamamoto}.
To state the deformation theory result we recall the following definition.

\begin{definition}\label{dfn:moduli}
For a 7-manifold $X$ with a coclosed $\G_2$-structure $\varphi$ and a Hermitian complex line bundle $L$ on $X$,  we let $\mathcal{M}(X,\varphi,L)$ denote the \emph{moduli space} of deformed $\G_2$-instantons on $L$ with respect to $\varphi$.  
Let $A\in\mathcal{M}(X,\varphi,L)$ and consider the complex
\begin{equation}\label{eq:complex}
0\longrightarrow \Omega^0(X) \stackrel{d}{\longrightarrow} \Omega^1(X) \stackrel{(\frac{1}{2}F_A^2+*\varphi)\wedge d}\longrightarrow d\Omega^5(X)\longrightarrow 0.
\end{equation}
Then $A$ is \emph{unobstructed} if  $H^2=0$ for   \eqref{eq:complex}, i.e.~if the linearisation of the deformed $\G_2$-instanton condition $(\frac{1}{2}F_A^2+*\varphi)\wedge d:\Omega^1(X)\to d\Omega^5(X)$ is surjective; otherwise $A$ is \emph{obstructed}. 
\end{definition}

We now state a deformation theory result that follows from \cite{KawaiYamamoto}, which motivates the definition of unobstructed in Definition \ref{dfn:moduli}.

\begin{theorem}\label{thm:KY}
Let $(X^7,\varphi)$ be a compact $7$-manifold with a coclosed $\G_2$-structure, and let $L$ be a Hermitian complex line bundle on $X$.  Then $\mathcal{M}(X,\varphi,L)$ has expected dimension   $b^1(X)$. Therefore, if $A\in\mathcal{M}(X,\varphi,L)$ is unobstructed, then $\mathcal{M}(X,\varphi,L)$ is a smooth manifold near $A$ of dimension $b^1(X)$.

Moreover, if $A\in\mathcal{M}(X,\varphi,L)$ and 
\begin{itemize}\item $d\varphi=\lambda \ast_\varphi \varphi$ for some $\lambda \in \mathbb{R}$, or 
 \item $F_A^3\neq 0$ everywhere,
\end{itemize}
 then for generic coclosed $\G_2$-structures $\varphi'$ sufficiently near $\varphi$ so that $[*_{\varphi'}\varphi']=[*_{\varphi}\varphi]\in H^4(X)$, the subset of $\mathcal{M}(X,\varphi',L)$ of connections sufficiently near $A$ is a smooth manifold of dimension $b^1(X)$ (if it is non-empty).
\end{theorem}

Theorem \ref{thm:KY} has the following  consequence. 

\begin{corollary}\label{cor:moduli}
 Let $(X^7,\varphi)$ be a compact $7$-manifold with a coclosed $\G_2$-structure, and suppose that $X^7$ admits a nearly parallel $\G_2$-structure.  Let $L$ be a Hermitian complex line bundle on $X$.  Then the expected dimension of $\mathcal{M}(X,\varphi,L)$ is $0$.  In particular, if $A\in\mathcal{M}(X,\varphi,L)$ is unobstructed then $A$ is rigid and locally isolated in $\mathcal{M}(X,\varphi,L)$. 
 
 Moreover, given $A\in\mathcal{M}(X,\varphi,L)$ such that $F_A^3\neq 0$ everywhere, for generic coclosed $\G_2$-structures $\varphi'$ near $\varphi$ with $[*_{\varphi'}\varphi']=[*_{\varphi}\varphi]$, the subset of $\mathcal{M}(X,\varphi',L)$ of connections sufficiently near $A$ is a discrete collection of points (if it is non-empty).
\end{corollary}

\begin{proof}
 The result follows from Theorem \ref{thm:KY} and the fact that $X^7$ must have finite fundamental group by Myers theorem, since  the induced metric from a nearly parallel $\G_2$-structure is Einstein with positive scalar curvature.
\end{proof}

Since the non-trivial deformed $\G_2$-instantons given by Corollary \ref{cor:examples.np} are not rigid for the nearly parallel $\G_2$-structures $\varphi^{np}$ and $\varphi^{ts}$ on the 3-Sasakian $X^7$,
Corollary \ref{cor:moduli} yields the following result.

\begin{proposition}
All of the non-trivial deformed $\G_2$-instantons from Propositions \ref{prop:examples} and \ref{prop:AW} that exist in positive-dimensional families are obstructed.

In particular, the non-trivial deformed $\G_2$-instantons on $(X^7,\varphi^{ts})$ and $(X^7,\varphi^{np})$ on the trivial line bundle $L_0$ from Proposition \ref{prop:examples} are obstructed.  
\end{proposition}

\begin{remark}
We see from \eqref{eq:Cubic_Term} and Propositions \ref{prop:examples} and \ref{prop:AW} that all of the non-trivial deformed $\G_2$-instantons $A$ with respect to $\varphi_{t,\varepsilon}$ we have constructed have the property that $F_A^3\neq 0$ everywhere, and thus Corollary \ref{cor:moduli} applies.  Moreover, when $\varepsilon=+1$ we see that $[\psi_{t,+1}]=[\psi_{t',+1}]$ for all $t,t'$.  However, $\mathcal{M}(X,\varphi_{t',+1},L_0)$ still contains an $S^2$ family of deformed $\G_2$-instantons near $A$ for all $t'$ near $t$.  We conclude that the family $\varphi_{t,+1}$ is not sufficiently generic to enable us to peturb $A$ to become locally isolated.
\end{remark}

We now make an elementary observation, which follows from Remark 5.13 in \cite{KawaiYamamoto}.

\begin{lemma}\label{lem:trivial.moduli}
 Let $(X^7,\varphi)$ be a compact $7$-manifold with a nearly parallel $\G_2$-structure, and let $L_0$ be the trivial complex line bundle on $X$.  Then the trivial flat connection is unobstructed as a deformed $\G_2$-instanton, thus rigid and locally isolated in $\mathcal{M}(X,\varphi,L_0)$.
\end{lemma}

Lemma \ref{lem:trivial.moduli} yields the following interesting result.

\begin{proposition} For any $3$-Sasakian $7$-manifold $X$, the moduli spaces $\mathcal{M}(X,\varphi^{np},L_0)$ and $\mathcal{M}(X,\varphi^{ts},L_0)$ have at least two components of different dimensions.
\end{proposition}

\begin{proof}
Since the trivial flat connection $A_0$ lies in $\mathcal{M}(X,\varphi^{np},L_0)$ and $\mathcal{M}(X,\varphi^{ts},L_0)$ and is locally isolated, it must define a component of the moduli space in each case.  Therefore, in each case the positive-dimensional family of non-trivial deformed $\G_2$-instantons from Proposition \ref{prop:examples} must lie in a different component of the moduli space to the trivial flat connection.
\end{proof}

\section{A Chern--Simons type functional}\label{sec:Chern-Simons}

In this section, we study a functional of Chern--Simons type, introduced in \cites{KarigiannisLeung}, which has deformed $\G_2$-instantons as critical points, in the setting of our examples.

\subsection{The functional} Let $X^7$ be a compact $7$-manifold with a coclosed $\G_2$-structure $\varphi$.  (The assumption of compactness is to ensure that integrals of various quantities over $X$ are guaranteed to be well-defined.)  Let $L$ be a Hermitian complex line bundle over $X$ and fix a unitary connection $A_0$ on $L$.  We let $t$ denote a coordinate on $[0,1]$ and we shall pullback various quantities from $X$ to $X\times [0,1]$ such as $\psi=*_{\varphi}\varphi$ and $L$; for ease of notation, we shall omit the pullback symbol.

Then, for any other unitary connection $A$ on $L$ we consider a unitary connection $\bA$ on the pullback of $L$ to  $[0,1] \times X$, given by 
\begin{equation}\label{eq:bA}
%A_t
\bA=A_0+t(A-A_0).
\end{equation} 
We let $A_t=\bA|_{\{t\}\times X}$ and let $F_t$ be the curvature of $A_t$, which is
\begin{equation}\label{eq:Ft}
F_t=F_{A_0}+t(F_A-F_{A_0}).
\end{equation}
  Hence, the curvature of $\bA$ in \eqref{eq:bA} can be written 
\begin{equation}\label{eq:bF}
\bF=dt\wedge (A-A_0)+F_t.
\end{equation}
With this notation, we can make the following definition.

\begin{definition}\label{dfn:functional} Let $(X^7,\varphi)$ be a compact 7-manifold with a coclosed $\G_2$-structure and let $L$ be a Hermitian complex line bundle on $X$.  Given a unitary connection $A_0$ on $L$  we define the functional $\mathcal{F}$ on unitary connections $A$ on $L$ over $X$ by the formula:
\begin{equation}\label{eq:functional}
\cF(A)= \int_{X \times [0,1]} e^{\bF + \psi},
\end{equation}
where $\psi=*_{\varphi}\varphi$ and $\bF$ is given in \eqref{eq:bF}. 
It is shown in \cites{KarigiannisLeung} that the deformed $\G_2$-instanton equation \eqref{eq:dG2.intro} arises as the critical point equation for the functional $\cF$.
\end{definition}

Since $X$ is 7-dimensional, it is convenient to expand and rewrite $\cF$ in \eqref{eq:functional} as follows:
\begin{equation}\label{eq:functional.2}
\cF(A)  = \int_{X \times [0,1]} \frac{1}{3!} (\bF + \psi)^3 + \frac{1}{4!} (\bF + \psi)^4 %\\& 
= \frac{1}{2} \int_{X \times [0,1]}  \bF^2 \wedge \psi + \frac{1}{12} \bF^4 .
\end{equation}
On the other hand, we have $\bF^k = F_{A_0}^k + d \left( cs_k(A_0,\bA) \right)$ on $[0,1] \times X$, where $\bF$ is given in \eqref{eq:bF} and $cs_k(A_0,\bA)$ is the $k$th transgression form. This can be explicitly written using the curvature $\bF_s$ of the connections $\bA_s= A_0 + s(\bA -A_0)$ for $s$ a real valued parameter. Indeed, we have
\begin{equation}\label{eq:csk}
cs_k(A_0,\bA)= k \int_0^1  (\bA-A_0) \wedge \bF_s^{k-1} \ ds .
\end{equation}
Thus, from Stokes' theorem and the fact that $F_{A_0}^2 \wedge \psi =0 = F_{A_0}^4$ by dimensional reasons, we may write \eqref{eq:functional.2} as:
\begin{equation}\label{eq:functional.3}
\cF(A)  = \frac{1}{2} \int_{X} cs_2(A_0,A) \wedge \psi + \frac{1}{12} cs_4(A_0,A) .
\end{equation}
In particular, we have the following observation.

\begin{lemma}\label{lem:functional.L0} In the notation of Definition \ref{dfn:functional}, let $L=L_0$ be the trivial complex line bundle and let $A_0$ be the trivial flat connection. Then the functional $\cF$ in \eqref{eq:functional} is given by
\begin{equation}\label{eq:functional.L0}
\cF(A) = \frac{1}{2} \int_{X} (A-A_0) \wedge \left( F_A \wedge \psi + \frac{1}{12} F_A^{3} \right) .
\end{equation}
\end{lemma}

\begin{proof}
The result follows from \eqref{eq:functional.3} and the observation that $cs_k(A_0,A)=(A-A_0) \wedge F_A^{k-1}$ by \eqref{eq:csk} as $F_{A_0}=0$.
\end{proof}

\subsection{Examples} For connections as in \eqref{eq:Connection} on 3-Sasakian 7-manifolds we find, after a lengthy but straightforward computation, the following using \eqref{eq:functional.L0}.

\begin{proposition}\label{prop:functional} Let $X^7$ be a $3$-Sasakian $7$-manifold with the coclosed $\G_2$-structure $\varphi_{t,\varepsilon}$ on \eqref{eq:G2_Structure}, and recall the $3$-Sasakian metric $g^{ts}$ in \eqref{eq:g.ts}.  For connections $A=ia_j\eta_j$ as in \eqref{eq:Connection} on the trivial complex line bundle $L_0$ on $X$, we have that the functional $\cF$ in \eqref{eq:functional} is given by
\begin{equation}\label{eq:functional.t.eps}
\cF(A)=- c \left[ (x^2+y^2)(2(x^2+y^2)-1) + 2 t^2 (  x^2+\varepsilon y^2 ) \right],
\end{equation}
where $x=a_3$, $y^2=a_1^2+a_2^2$ and
\begin{equation}\label{eq:c}
c= \int_X \eta_{123} \wedge \pi^* \vol_Z= \mathrm{Vol}(X,g^{ts})>0.
\end{equation}
\end{proposition}
\begin{remark}
Notice that $\cF$ in \eqref{eq:functional.t.eps} is bounded above and that, as can be seen from \eqref{eq:c}, the 3-Sasakian metric $g^{ts}$ may be rescaled so that $c=1$, which we will now do for convenience. 
\end{remark}

\begin{remark}
The critical points of the functional $\cF$ in \eqref{eq:functional.t.eps}, restricted to connections given by the ansatz \eqref{eq:Connection}, are given by the vanishing of
	\begin{align}
	\frac{\partial \cF}{\partial x} & = -2x \left( 4 (x^2+y^2) + 2 t^2 -1\right), \label{eq:crit.pt.1}\\
	\frac{\partial \cF}{\partial y} & = -2y \left( 4 (x^2+y^2) + 2\epsilon t^2 -1\right) .\label{eq:crit.pt.2}
	\end{align}
We see that we have solutions to \eqref{eq:crit.pt.1}--\eqref{eq:crit.pt.2} given by $x=0=y$, corresponding to the flat connection, and 
\begin{equation}\label{eq:crit.pt.+1}
x^2+y^2= \frac{1}{4} (1-2t^2)  \ \text{when $\epsilon =1$ and } 2t^2<1 ,
\end{equation}
	or $\epsilon=-1$ and either 
\begin{equation}\label{eq:crit.pt.-1}
x=0 \text{ and } y^2=\frac{1}{4}(1+2t^2), \quad \text{or} \quad y=0 \text{ and } x^2=\frac{1}{4}(1-2t^2) \ \text{when } 2t^2<1 . 
\end{equation}
Equations \eqref{eq:crit.pt.+1}--\eqref{eq:crit.pt.-1} coincide with the conditions \eqref{eq:constraint}--\eqref{eq:constraint.3} we derived earlier that gave our non-trivial deformed $\G_2$-instantons in Proposition \ref{prop:examples}.
\end{remark}

Using Proposition \ref{prop:functional} we can examine the relationship between the trivial flat connection and the functional $\cF$.  In particular, we see that the nature of the trivial connection as a critical point for $\cF$ depends on the choice of $\G_2$-structure.

\begin{lemma}\label{lem:flat.crit} Recall the notation of Proposition \ref{prop:functional}, in particular the functional $\cF$ in \eqref{eq:functional.t.eps}. Let $A_0$ be the trivial flat connection on the trivial complex line bundle $L_0$ over $X$.
Then the Hessian of $\cF$ is nondegenerate at $A_0$ if and only if $t\neq 1/\sqrt{2}$.  Moreover: 
\begin{itemize}
\item if $t\in (0,1/\sqrt{2})$ then $A_0$ is a local minimum of $\cF$;
\item if $t\geq 1/\sqrt{2}$ and $\varepsilon=+1$ then $A_0$ is a local maximum of $\cF$;
\item if $t\geq 1/\sqrt{2}$ and $\varepsilon=-1$ then $A_0$ is a saddle point of $\cF$. 
\end{itemize}
\end{lemma}

\begin{proof}
By direct computation, one can calculate the Hessian of the functional $\cF$ in \eqref{eq:functional.t.eps}.  At the flat connection $A_0$, when $(x,y)=(0,0)$, the Hessian of $\cF$ has eigenvalues 
\begin{equation}\label{eq:eigenvalues}
2(1-2t^2)\quad\text{and}\quad 2(1-2\varepsilon t^2 ).
\end{equation} 
We see immediately that the Hessian of $\cF$ is degenerate if and only if $2t^2=1$ as claimed.  Moreover, as long as $2t^2\neq 1$, the critical point is characterised by the signs of the eigenvalues in \eqref{eq:eigenvalues}. 
When $2t^2=1$ we see that 
\begin{equation}\label{eq:functional.special}
\cF(A)=-2(x^2+y^2)^2+(1-\varepsilon)y^2.
\end{equation}
When $\varepsilon=+1$, $\cF\leq 0$ and equals $0$ if and only if $x=y=0$.  When $\varepsilon=-1$, we instead see that inserting $x=0$ in \eqref{eq:functional.special} gives a function of $y$ with a local minimum at $y=0$, and for $y=0$ in \eqref{eq:functional.special} we obtain a local maximum at $x=0$.  The result then follows.
\end{proof}

We already observed the significance of the value $t=1/\sqrt{2}$ in Proposition \ref{prop:examples}. Lemma \ref{lem:flat.crit} now leads to the following additional observation concerning this value of $t$.

\begin{corollary}\label{cor:flat.obs}
The trivial flat connection on the trivial complex line bundle over $X$ is obstructed as a deformed $\G_2$-instanton for the $\G_2$-structures $\varphi_{1/\sqrt{2},\varepsilon}$ in \eqref{eq:G2_Structure} for $\varepsilon=\pm 1$.
\end{corollary}

\begin{proof}
When $t=1/\sqrt{2}$, Lemma \ref{lem:flat.crit} shows that the Hessian of $\cF$ in \eqref{eq:functional.t.eps} is degenerate at the trivial flat connection $A_0$.  Thus, there exist non-trivial infinitesimal deformations of $A_0$ as a deformed $\G_2$-instanton within the ansatz \eqref{eq:Connection}, i.e.~$H^1$ of the complex \eqref{eq:complex} is non-zero.  

However, we know from Proposition \ref{prop:examples} that $A_0$ is locally isolated as a deformed $\G_2$-instanton for $t=1/\sqrt{2}$ amongst those given by \eqref{eq:Connection}, and so the non-trivial infinitesimal deformation must be obstructed, i.e.~$H^2$ of the complex \eqref{eq:complex} must also be non-zero. 
\end{proof}

\begin{remark}
 The fact that the flat connection is obstructed at $t=1/\sqrt{2}$ was to be expected, as it is at this point that the transition occurs when the $2$-sphere or two points consisting of non-flat deformed $\G_2$-instantons ``shrinks'' and merges with the flat connection (see Figure \ref{fig:moduli}). Nevertheless, this observation contrasts with the case of nearly parallel $\G_2$-structures for which the flat connection is always unobstructed (see Lemma \ref{lem:trivial.moduli}). 
\end{remark}

To conclude, we compare the functional $\cF$ in \eqref{eq:functional.t.eps} for the pairs of isometric $\G_2$-structures $\varphi_{1/\sqrt{5}, \epsilon}$ and $\varphi_{1,\epsilon}$, for $\epsilon\in\{\pm 1\}$, which induce the Einstein metrics $g^{np}$ and $g^{ts}$ respectively. 

\begin{example} 
 The coclosed $\G_2$-structures $\varphi_{1/\sqrt{5},+1}=\varphi^{np}$ and $\varphi_{1/\sqrt{5}, -1}$ both induce the strictly nearly parallel metric $g^{np}$ on the 3-Sasakian 7-manifold $X$, which has the property that the metric cone on $(X,g^{np})$ has holonomy $\Spin(7)$. These $\G_2$-structures determine rather different functionals $\cF$ by Proposition \ref{prop:examples}. Figure \ref{fig:eps.np} plots the functional $\cF$, restricted to connections given by the ansatz \eqref{eq:Connection}, as in \eqref{eq:functional.t.eps} for these two $\G_2$-structures. 
\vspace{-10pt} 
 
\begin{figure}[H]
\centering
	\includegraphics[width=0.48\textwidth,height=0.3\textheight]{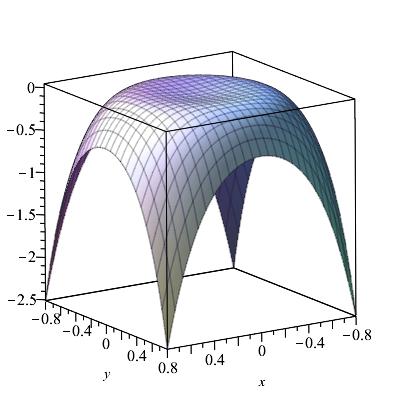}
	\quad\includegraphics[width=0.48\textwidth,height=0.3\textheight]{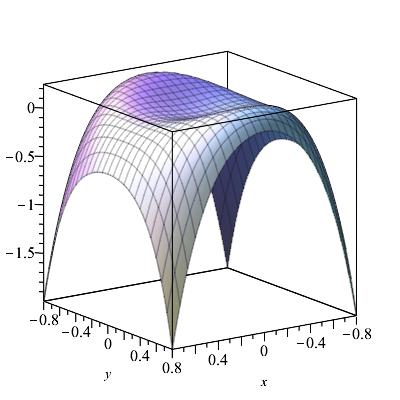}
	\caption{The functional $\cF$ for $(t,\varepsilon)=(1/\sqrt{5},+1)$ and $(t,\varepsilon)=(1/\sqrt{5},-1)$.}\label{fig:eps.np}
\end{figure}

\noindent One can just discern the local minimum at the origin in the left-hand plot in Figure \ref{fig:eps.np}, predicted by Lemma \ref{lem:flat.crit}. We also illustrate the difference between these cases via their levels sets in Figure \ref{fig:lev.np}.  One can see the circle of critical points (which are local maxima) for $\varepsilon=+1$ which gives the 2-sphere of non-trivial deformed $\G_2$-instantons in Proposition \ref{prop:examples}, in contrast to the two pairs of critical points for $\varepsilon=-1$ which give the circle (given by local maxima) and two further examples (which are saddle points) of non-trivial deformed $\G_2$-instantons in Proposition \ref{prop:examples}.

\begin{figure}[H]
	\centering
	\includegraphics[width=0.48\textwidth,height=0.3\textheight]{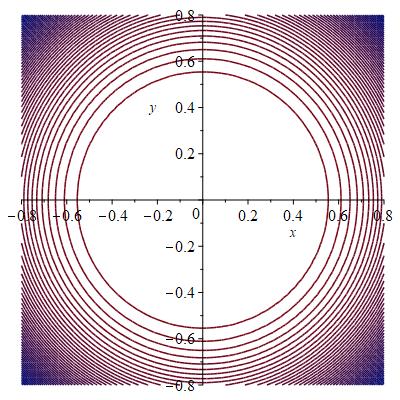}
\quad
	\includegraphics[width=0.48\textwidth,height=0.3\textheight]{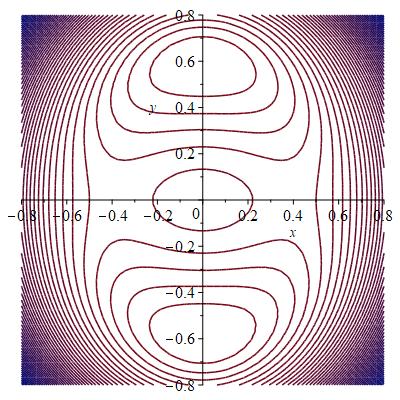}
	\caption{Level sets of $\cF$ for $(t,\varepsilon)=(1/\sqrt{5},+1)$ and $(t,\varepsilon)=(1/\sqrt{5},-1)$.}\label{fig:lev.np}
\end{figure}

\end{example}

\begin{example} 
	We now focus on the coclosed $\G_2$-structures $\varphi_{1, \epsilon}$, for $\epsilon=\pm 1$, recalling that $\varphi^{ts}=\varphi_{1,-1}$. These $\G_2$-structures induce the $3$-Sasakian metric $g^{ts}$ on $X$, which is that whose metric cone is hyperk\"ahler. In this case, we again already know the functionals $\cF$ for these two $\G_2$-structures are very different by Proposition \ref{prop:examples}, and further evidence is provided by the following plots of the functional $\cF$ in \eqref{eq:functional.t.eps} in Figure \ref{fig:eps.ts}.
	
	\vspace{-10pt} 
	
	\begin{figure}[H]
	\centering
	\includegraphics[width=0.48\textwidth,height=0.3\textheight]{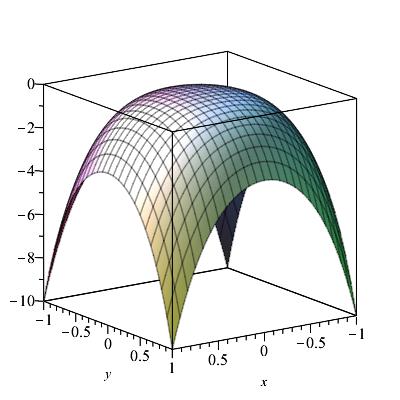}
\quad \includegraphics[width=0.48\textwidth,height=0.3\textheight]{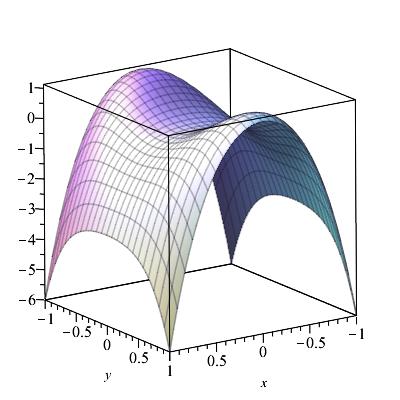}
		\caption{The functional $\cF$ for $(t,\varepsilon)=(1,+1)$ and $(t,\varepsilon)=(1,-1)$.}\label{fig:eps.ts}
	\end{figure}
	
	As for the level sets of the functional $\cF$, these are plotted in Figure \ref{fig:lev.ts} for each case.  For $\varepsilon=+1$ we see that the only critical point is the origin,  giving the trivial flat connection, and instead we have a pair of critical points (which are local maxima) for $\varepsilon=-1$ which define a circle of deformed $\G_2$-instantons from Proposition \ref{prop:examples}.
	
	\begin{figure}[H]
\includegraphics[width=0.48\textwidth,height=0.3\textheight]{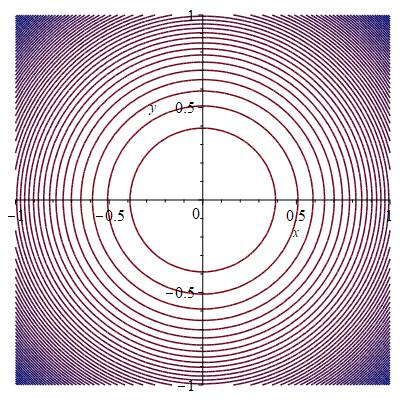}
	\quad
	\includegraphics[width=0.48\textwidth,height=0.3\textheight]{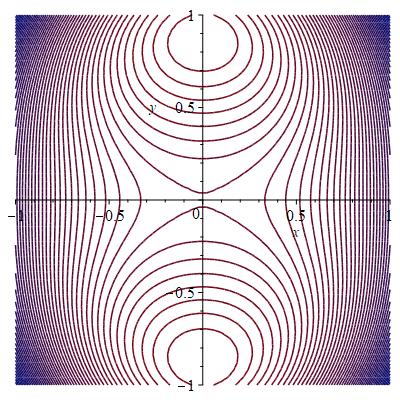}
		\caption{Level sets of $\cF$ for $(t,\varepsilon)=(1,+1)$ and $(t,\varepsilon)=(1,-1)$.}\label{fig:lev.ts}
	\end{figure}
	
\end{example}

\bibliographystyle{plain}	
%===============================================================================

%\bibliography{refsa1}

%===============================================================================

\end{document}